\documentclass[11pt,oneside,english,jou]{amsart}
\usepackage[T1]{fontenc}
\usepackage[latin9]{inputenc}
\usepackage{babel}
\usepackage{verbatim}
\usepackage{mathrsfs}
\usepackage{amstext}
\usepackage{amsthm}
\usepackage{amssymb}
\usepackage{amsfonts}
\usepackage{amsmath}
\usepackage{esint}
\usepackage{enumerate}
\usepackage{relsize} 
\usepackage[unicode=true,pdfusetitle,
bookmarks=true,bookmarksnumbered=false,bookmarksopen=false,
breaklinks=false,pdfborder={0 0 1},colorlinks=false]
{hyperref}
\usepackage[margin=2.5cm]{geometry}
\usepackage[foot]{amsaddr}
\usepackage{mathtools}
\usepackage{ dsfont }
\usepackage[shortlabels]{enumitem}
\usepackage{bm}

\makeatletter
\numberwithin{equation}{section}
\numberwithin{figure}{section}
\theoremstyle{plain}
\newtheorem{thm}{\protect\theoremname}[section]
\theoremstyle{definition}
\newtheorem{rem}[thm]{\protect\remarkname}
\theoremstyle{definition}
\newtheorem{defn}[thm]{\protect\definitionname}
\theoremstyle{plain}
\newtheorem{prop}[thm]{\protect\propositionname}
\theoremstyle{plain}
\newtheorem{lem}[thm]{\protect\lemmaname}
\theoremstyle{plain}
\newtheorem{conjecture}{\protect\conjecturename}
\theoremstyle{plain}

\theoremstyle{definition}

\theoremstyle{definition}

\theoremstyle{definition}

\theoremstyle{definition}

\newtheorem{definition}{Definition}

\usepackage{tikz}
\usetikzlibrary{shapes,arrows}
\usepackage{verbatim}
\usepackage{amsthm}
\usepackage{amstext}

\newcommand{\R}{\mathbb R}
\newcommand{\Z}{\mathbb Z}
\newcommand{\N}{\mathbb N}

\newcommand{\PP}{\mathbb P}
\def\RP{\R\PP}

\newcommand{\eps}{\varepsilon}

\makeatother

\providecommand{\conjecturename}{Conjecture}
\providecommand{\corollaryname}{Corollary}
\providecommand{\definitionname}{Definition}
\providecommand{\examplename}{Example}
\providecommand{\lemmaname}{Lemma}
\providecommand{\problemname}{Problem}
\providecommand{\propositionname}{Proposition}
\providecommand{\remarkname}{Remark}
\providecommand{\theoremname}{Theorem}
\providecommand{\taskname}{Task}

\def\Ok{{\mathcal O}}
\def\Dh{\dim_{\rm H}}

\def\half{\frac{1}{2}}

\def\Lam{\Lambda}
\newcommand{\gam}{\gamma}

\newcommand{\sig}{\sigma}

\newcommand{\bi}{{\bf i}}
\newcommand{\bj}{{\bf j}}

\def\N{{\mathbb N}}

\def\Ak{{\mathcal A}}

\def\Ik{{\mathcal I}}

\def\be{\begin{equation}}
	\def\ee{\end{equation}}

\newcommand{\Ek}{{\mathcal E}}
\newcommand{\Fk}{{\mathcal F}}
\def\Gk{{\mathcal G}}

\def\ov{\overline}

\def\wt{\widetilde}

\begin{document}

\title{On nonlinear iterated function systems with overlaps}


\date{\today}

\author{Boris Solomyak }
\address{Boris Solomyak\\ Department of Mathematics,
Bar-Ilan University, Ramat-Gan, Israel}
\email{bsolom3@gmail.com}

\thanks{Supported by the Israel Science Foundation grant \#1647/23}

\begin{abstract}
We construct an example of an
iterated function system on the line, consisting of linear fractional transformations, such that two of the maps share a fixed points, but the dimension of the attractor
equals the conformal dimension, so that there is no ``dimension drop''.
\end{abstract}

\maketitle

\thispagestyle{empty}

\section{Introduction}

Let $\Phi=\{\phi_i\}_{i\in \Ik}$ be a collection of $C^{1+\theta}$-smooth maps of a compact interval $I\subset \R$ into itself. Here $\Ik$ is a finite alphabet, with $\#\Ik\ge 2$.
We call $\Phi$ an {\em iterated function system} (IFS). The IFS is assumed to be {\em hyperbolic}, that is, there exist $0 < \gam_1 < \gam_2 < 1$, such that
$$
0 < \gam_1 \le |\phi_i'(x)| \le \gam_2 < 1\ \ \ \mbox{for all}\ i\in \Ik\ \ \mbox{and}\ x\in I.
$$
By \cite{Hutch}, there exists a unique non-empty compact set $\Lam = \Lam_\Phi$, called the {\em attractor} of the IFS, such that $\Lam = \bigcup_{i\in \Ik} \phi(\Lam)$.
Such an IFS is called {\em self-conformal}, and the attractor is sometimes called a {\em self-conformal set}.
It was shown by Falconer \cite{Falc89} (see also \cite[Chapter 3]{Falconer_Tech}) that the Hausdorff dimension of a self-conformal set is always equal to its box-counting (Minkowski) dimension,
so we will simply write $\dim(\Lam)$ below.  We are interested in computing this dimension.

For a finite word $u\in \Ik^n$ let
$$
\phi_u := \phi_{u_1}\circ \cdots \circ \phi_{u_n},\ \ \ I_u:= \phi_u(I).
$$
The {\em pressure function} is given by
\be \label{eq:pressure}
P_\Phi(t) = \lim_{n\to \infty} \frac{1}{n}\log \sum_{u\in \Ik^n} \|\phi'_u\|^t,
\ee
where $\|\cdot\|$ is the supremum norm on $I$. Using the Bounded Distortion Property, it is easy to see that the limit in \eqref{eq:pressure} exists, and there is a unique zero of
{\em Bowen's equation}
$$
P_\Phi(s) = 0,
$$
whose solution $s=s(\Phi)$ is called the {\em conformal dimension} of $\Phi$. We always have the upper bound, obtained using the natural covers by cylinder intervals $I_u$, with
$|u|=n$:
$$
\dim(\Lam_\Phi) \le \min\{1,s(\Phi)\}.
$$
If the Open Set Condition holds, that is, there exists a nonempty open set $\Ok$ such that $f_i(\Ok) \subset \Ok$ and all $f_i(\Ok)$ are mutually disjoint, then
$$
\dim(\Lam_\Phi) = s(\Phi).
$$
These are all classical results; see the original papers \cite{Ruelle1982,Bowen1979}, as well the books \cite{Falconer_Tech,PUbook,BSSbook}. 
There have been many extensions and generalizations, in particular,
to infinite hyperbolic IFS \cite{MU1996} and to parabolic IFS \cite{MU2000}, but we do not discuss them here.

What happens in the ``overlapping case,'' e.g., when the Open Set Condition fails (or when it is not obvious whether it holds or not)? This is much less understood,  
although the dimension
properties of such IFS have been studied for a long time. We do not survey the extensive literature on this topic, but mention the important progress, which occurred in the last ten years
in the study of {\em self-similar} IFS's, that is, when the maps $\phi_i(x) = r_i x + a_i$, with $|r_i|\in (0,1)$ are affine linear contractions. In the self-similar case the conformal dimension
is known as the {\em similarity dimension}:
$$
\dim_{\rm sim}(\Phi) = s,\ \ \ \mbox{where}\ \ \Phi = \{x\mapsto r_i x + a_i\}_{i\in \Ik},\ \ \sum_{i\in \Ik} r_i^s = 1.
$$
The following conjecture was first stated by 
Simon \cite{Si96} in this generality, although special cases of it have been considered earlier.

\medskip

\begin{conjecture}[Exact coincidence conjecture] \label{conj1} If $\Phi = \{x\mapsto r_i x + a_i\}_{i\in \Ik},\ x\in \R$, with $|r_i|\in (0,1)$, is such that
$$ 
\dim(\Lam_\Phi) < \min\{1,\dim_{\rm sim}(\Phi)\},
$$
then $\Phi$ has an ``exact overlap,'' that is, the semigroup generated by $\Phi$ is not free. (Equivalently, there exists $n\in \N$ and two distinct words $u,v\in \Ik^n$, such that
$\phi_u \equiv \phi_v$.)
\end{conjecture}

\medskip

Although the general conjecture remains open, Hochman \cite{Hochman2014}, using ideas from additive combinatorics, has achieved a major breakthrough in this direction. 
The next definition is from \cite{SolTak}.


\begin{defn} \label{def-sep}
Let $\Fk = \{f_i\}_{i\in \Ik}$ be an IFS on a metric space $(X,\varrho)$, that is, $f_i:X\to X$. We say that $\Fk$ 
satisfies the \emph{exponential separation condition} on a set $X'\subseteq X$
if there exists $c > 0$  such that for all $n\in \N$ sufficiently large we have
\begin{equation}\label{exp_sep}
\sup_{x\in X'} \varrho(f_\bi(x), f_\bj(x) ) > c^{n},\ \ \mbox{for all}\  \bi,\bj\in \Ik^{n}\ \ \mbox{with}\  \ f_\bi\not \equiv f_\bj.
\end{equation}
If, in addition, the semigroup generated by $\Fk$ is free, that is, $f_\bi \equiv f_\bj \ \Longleftrightarrow\ \bi=\bj$, we say that $\Fk$ satisfies the {\em strong exponential separation condition}, abbreviated SESC.
If these properties hold for infinitely many $n$, then we say that $\Fk$ satisfies the {\em (strong) exponential separation condition on $X'$ along a subsequence}. 
\end{defn}

\begin{thm}[{Hochman \cite[Corollary 1.2]{Hochman2014}}] \label{th:hochman}
If a self-similar IFS $\Phi = \{x\mapsto r_i x + a_i\}_{i\in \Ik},\ x\in \R$, with $|r_i|\in (0,1)$, satisfies the SESC along a subsequence, then
$$ 
\dim(\Lam_\Phi) = \min\{1,\dim_{\rm sim}(\Phi)\},
$$
\end{thm}

This already implies the validity of Conjecture~\ref{conj1} for self-similar IFS with algebraic parameters. In fact, the results for self-similar sets are deduced from
results on self-similar measures, and there is an analogous conjecture for the dimension of such, see \cite{Hochman2014}, but here we restrict ourselves to the dimension of
attractors for brevity. Further progress, building on the work of Hochman \cite{Hochman2014}, was achieved by Varj\'u \cite{Varju} and Rapaport \cite{Rapaport}, but the conjecture is
still open.

\medskip

Now we turn to the overlapping nonlinear case, where our knowledge is much more limited. Some results of ``almost every''  type were obtained in \cite{SSU1,SSU2} 
(for parabolic and infinite hyperbolic IFS on the line as well) and in
 some later papers (we do not provide an exhaustive bibliography here), using the {\em transversality method}. However, this method has many limitations, since the transversality condition is difficult to check, and moreover, quite often it is known to fail.
 
 One may wonder whether some form of the Exact Coincidence Conjecture holds for a class of nonlinear IFS, maybe under extra assumptions, like real analyticity of the maps.
 Notice that in the non-real analytic case the exact coincidence should be defined for the restrictions of the iterates to the attractor, i.e.,
 $$
 \phi_u|_{\Lam_\Phi} \equiv \phi_v|_{\Lam_\Phi}\ \ \mbox{for some}\ u\ne v \ \mbox{in \ $\Ik^n$}.
 $$
 
 It would be very interesting to extend Hochman's Theorem \ref{th:hochman} in some form to a class of non-linear IFS. As far as we aware, this is currently known only for IFS consisting of 
 linear-fractional transformations. The following result was obtained in joint work with Takahashi \cite{SolTak}.
  We quote \cite[Corollary 1.10]{SolTak} (in a special case for simplicity).

\begin{thm}[{\cite{SolTak}}] \label{th1} \label{th:sotak1}
Let $\mathcal{F} = \{ f_i \}_{ i \in \Ik }$ be a finite collection of linear fractional transformations with real coefficients.  
Assume that there exists $U\subset \R$, a bounded open interval, such that 
$f_i( \overline{U} ) \subset U$ for all $i \in \Ik$.  Let $\Lam_\Fk$ be the attractor of the IFS $\Fk$, and assume that $\Lam_\Fk$ is not a singleton.
If $\mathcal{F}$ satisfies the strong exponential separation condition (SESC) on a non-empty subset of $\ov{U}$, then   
$
\Dh(\Lam_\Fk) = \min\{ 1, s(\Fk) \}, 
$
where $s(\Fk)$ is the conformal dimension of the IFS.
\end{thm}

Note that $\Lam_\Fk = \{x\}$ is equivalent to $x$ being the common fixed point of all
the maps of the IFS. In the case of a self-similar IFS this implies that the maps commute, which yields exact overlaps. This is no longer true for linear-fractional transformations, 
and thus we need the assumption of $\Lam_\Fk$ not being a singleton.

\begin{rem}
Theorem~\ref{th:sotak1} was obtained as a by-product of a study of {\em hyperbolic projective IFS} consisting of M\"obius transformations in \cite{SolTak}. It is, in fact, a rather 
straightforward consequence of the results and techniques from a joint work with Hochman \cite{HS2017} on the 
dimension properties of the Furstenberg (stationary) measure for random walks on the $SL(2,\R)$ in \cite{HS2017}. Some of the results from
\cite{SolTak} were extended by Christodoulou and Jurga \cite{ChrisJurga20} to projective IFS which contain a parabolic map.
\end{rem}

The following question was raised by Micha\l \ Rams in a personal communication to Bal\'azs B\'ara\'ny.

\medskip

\noindent {\bf Question.} {\em 
Is there a conformal (strictly contractive) IFS on $\R$, such that two maps share a fixed point, but nevertheless the dimension of the attractor is equal to the conformal 
dimension?}

\medskip

Note that this is impossible for a self-similar IFS, since two affine linear maps on $\R$, sharing a fixed point, necessarily commute.
In the next section we provide an affirmative answer, using Theorem~\ref{th:sotak1}.

\section{Example}
 
It is well-known that the action of $SL(2,\R)$ on $\RP^1$, which can be identified with $\R\cup  \{\infty\}$, can be expressed in terms of linear fractional transformations.
For 
\begin{equation*}
A = 
\begin{pmatrix}
a & b \\
c & d \\
\end{pmatrix}
\in SL_2(\mathbb{R}), 
\end{equation*}
let
$f_A(x) = (ax + b)/(cx + d)$.

We will consider the IFS $\{f_1,f_2,f_3\} = \{f_A, f_B, f_C\}$ on the real line, where
\be \label{ifs1}
A = \begin{pmatrix}
\half & 0 \\
2 & 2 \\
\end{pmatrix},\ \ \ \ \ 
B = \begin{pmatrix}
\half & 0 \\
0 &  2 \\
\end{pmatrix},\ \ \ \ \ 
C_t = \begin{pmatrix}
\half & t \\
0 & 2 \\
\end{pmatrix}
\ee
for an appropriate parameter $t$, so that
$$
f_1(x) = \frac{x}{4x+4},\ \ \ \ f_2(x) = \frac{x}{4},\ \ \ \ f_3^{(t)}(x) = \frac{x}{4}+ \frac{t}{2}.
$$
Notice that $f_1(0) = f_2(0) = 0$. We will require that $t>0$. Then $\{f_1,f_2,f^{(t)}_3\}$ is a strictly contracting conformal IFS on $I=I_t:=[0, x_0^{(t)}]$, 
where $x_0^{(t)} = 2t/3$ is the fixed point of $f_3^{(t)}$. Under these assumptions,
 $$\max_{j\le 3} \bigl\{\bigl\|{f_j'|}_I\bigr\|_\infty\bigr\} \le 1/4,$$ hence the conformal dimension of the IFS is strictly less than 1. Notice that
 $$
 f_1(I) = [0, t/(4t+6)],\ \ \ \ f_2(I) = [0, t/6],\ \ \ \ f_3^{(t)}(I) = [t/2,x_0^{(t)}] = [t/2, 2t/3].
 $$
 
 \begin{thm}\label{th:main}
 Consider the IFS $\Fk_t = \{f_1,f_2,f^{(t)}_3\}$, and let $\Lam_t$ be its attractor.
 For all $t>0$ outside a set of Hausdorff dimension zero holds
 $$
 \dim(\Lam_t) = \dim_{\rm conf}(\Fk_t).
 $$
 \end{thm}

It is known (and not hard to see) that for an IFS of linear-fractional transformations
 the SESC on a set containing at least three points is equivalent to the strong Diophantine condition for the corresponding sub-semigroup of $SL(2,\R)$.

\begin{definition}[{\cite{SolTak} and \cite{HS2017}}]
Let $\Ak = \{A_i\}_{i\in \Ik}$ be a finite subset of a semi-simple Lie group $G$ equipped with a metric $\varrho$. Write $A_\bi = A_{i_1}\cdots A_{i_n}$ for $\bi = i_1\ldots i_n$.
We say that the set $\Ak$ is {\em Diophantine} if there exists a constant $c>0$ such that for every $n\in \N$, we have
 \begin{equation} \label{Dioph1}
\bi,\bj\in \Ik^n,\ A_\bi\ne A_\bj \implies \varrho(A_\bi,A_\bj) > c^n.
\end{equation}
The set $\Ak$ is {\em strongly Diophantine} if there exists $c>0$ such that for all $n\in \N$,
\begin{equation} \label{Dioph2}
\bi,\bj\in \Ik^n,\ \bi\ne \bj \implies \varrho(A_\bi,A_\bj) > c^n.
\end{equation}
\end{definition}

If we could choose $t>0$ algebraic in such a way that the semigroup $\{A,B,C_t\}^+$ is free, we would be done, because a subset of $SL(2,\R)$ consisting of matrices with algebraic
elements is Diophantine (see, for example, \cite[Prop. 4.3]{GJS1999}). We do not know if this is possible. Instead, we start with the following ``warm-up''

\begin{lem} \label{lem:free}
For all but countable many $t>0$ the semigroup $\{A,B,C_t\}^+$ is free.
\end{lem}

\begin{proof}
First we note that $\{A,B\}^+$ is free. In general, the problem of freeness for semigroups of rational matrices is undecidable, even for triangular $2\times 2$ matrices \cite{CHK99}.
However, there are many effective sufficient conditions. In particular, it follows from \cite[Section 4.2]{CHK99} that $\{A,B\}^+$ is free. In fact,
 $\{A,B\}$ is easily reduced to the canonical form
$$\left\{ \begin{pmatrix}
1/4 & 0 \\
0 & 1 \\
\end{pmatrix},\ \ 
\begin{pmatrix}
1/4 & 1 \\
0 & 1 \\
\end{pmatrix}\right\}
$$ considered in that paper, and then \cite[Proposition 2]{CHK99} gives the claim. (For completeness, we include the proof of  freeness in the Appendix, following \cite{CHK99}).

Now suppose that $\{A,B,C_t\}^+$ is not free for some $t$. Then  $AX(t) = BY(t)$, where $X,Y$ are some words in $\{A,B,C_t\}^+$. (Indeed, $AX(t) = CY(t)$ and
$BX(t) = CY(t)$ are impossible, since $f_C(I)$ is disjoint from $f_A(I)$ and $f_B(I)$.) Then $u\mapsto AX(u) - BY(u)$ is a  matrix-function having a zero at $u=t$. This function
is well-defined and real-analytic on $(-\eps,+\infty)$, hence it is either constant zero, or it has no accumulation points in $[0,\infty)$. However, if it is constant zero, we obtain that
$AX(0) = BY(0)$. Note that $C_0 = B$, so $AX(0) = BY(0)$ then gives a non-trivial relation in the semigroup $\{A,B\}^+$, contradicting the claim above.

Thus, for any finite words $X,Y \in \{A,B,C_t\}^+$, there are at most countable many $t>0$ such that $AX(t) = BY(t)$, and the lemma is proved.
\end{proof}

In the end, our proof will not rely on Lemma~\ref{lem:free}. Instead
we are going to use another result from \cite{SolTak}, saying that a 1-parameter family of real-analytic IFS, with a real-analytic dependence on the parameter, under some mild 
non-degeneracy assumptions satisfies the SESC for all parameters outside of a Hausdorff dimension zero set. This is a generalization of \cite[Theorem 5.9]{Hochman2014}, which deals with real-analytic families of self-similar IFS.

\medskip 

 We recall the set-up from \cite[Section 2.3]{SolTak} in our special case (in \cite{SolTak} the case of  $d$-dimensional IFS is considered).
Let $J$ be a compact interval in $\R$ and $V$ a bounded open set in $\R$. 
Let $\Ik$ be a finite set, $|\Ik|\ge 2$, and suppose that for each $i \in \Ik$ we are given a real-analytic function
$$
\phi_i:\,\ov{V}\times J \to V. 
$$
This means that it is real-analytic on some neighborhood of $\ov{V}\times J$. Denote $\Phi_t = \{ \phi_i(x,t) \}_{ i \in \Ik }$. This is a real-analytic IFS on $\ov{V}$, depending on the parameter $t\in J$ real-analytically.
For $t\in J$ let $\Pi_t: \Ik^\N\to \R$ be the natural projection map corresponding to $\Phi_t$.

\begin{defn}
The IFS family $\Phi_t, \ t\in J$, is called {\em non-degenerate} if 
$$
\bi,\bi\in \Ik^\N,\ \bi\ne \bj \implies \exists\, t_0 \in J\ \ \mbox{such that}\  \Pi_{t_0}(\bi) \ne \Pi_{t_0}(\bj).
$$
\end{defn}

\begin{thm}[{\cite[Theorem 2.10]{SolTak}}] \label{th2}
Suppose that the family $\Phi_t, \ t\in J$, is non-degenerate. Then $\Phi_t$ satisfies the SESC condition on the singleton $\{x_0^{(t)}\}$ for any
fixed $x_0^{(t)}\in V$ for all parameters $t\in J$ outside of
an exceptional set of Hausdorff dimension zero.
\end{thm}

Our family $\Fk_t:= \{f_1,f_2, f_3^{(t)}\}$ 
is clearly NOT non-degenerate since for any $\bi, \bj\in \{1,2\}^\N$ we have $\Pi_{t}(\bi) = \Pi_{t}(\bj) = 0$. The idea is to consider a sequence of IFS:
\be \label{def:ftn}
\Fk_t^{(n)} = \bigl\{f^{(t)}_u:\ u\in \{1,2,3\}^n, \ u\ \mbox{contains  3}\bigr\}.
\ee
It is obvious that its attractor $\Lam_t^{(n)} \subset \Lam_t$, and we will show that $\dim_{\rm conf}( \Fk_t^{(n)})\to \dim_{\rm conf}( \Fk_t)$, as $n\to \infty$.
Since the upper bound $\dim(\Lam_t) \le \dim_{\rm conf}( \Fk_t)$ always holds, it is enough to show that for each $n$ the IFS $\Fk_t^{(n)}$ satisfies the SESC, and then by
Theorem~\ref{th1} we have $\dim(\Lam_t^{(n)}) = \dim_{\rm conf}( \Fk_t^{(n)})$, and the desired claim for $\Fk_t$ follows.

\begin{prop}\label{prop:nondegen}
For each $n\in \N$ there exist $0 < t_n < t_n' < \infty$ such that $\Fk_t^{(n)}$  is non-degenerate on $[t_n, t_n']$.
\end{prop}

\begin{proof}
It will actually be more convenient to prove the claim for 
$$
\wt \Fk_t^{(n)} = \bigl\{f^{(t)}_u:\ u = v3,\ |u|\le n,\ v \in \{1,2\}^*\bigr\}.
$$
Showing non-degeneracy for $\wt\Fk_t^{(n)}$ for all $n$ is sufficient, since every word in the symbolic space corresponding to $\Fk_t^{(n)}$ can be written as a concatenation of
words of length $\le 2n-1$ having a single $3$ at the end.

Denote by $I^t_u$ the cylinder interval for $\Fk_t$ corresponding to $u\in \Ak^*$, that is, $I^t_u = f^{(t)}_u(I)$. The following lemma is immediate.

\begin{lem} \label{lem1}
Suppose that for all $v,w\in \{1,2\}^*,\ v\ne w$, of length $\le n-1$, there exists $t>0$ such that $I^t_{v3} \cap I^t_{w3} = \emptyset$, then $\wt\Fk_t^{(n)}$
is non-degenerate.
\end{lem}

Of course, the parameter $t$ will usually be different for different pairs of cylinder intervals.

\medskip

We will use the following notation: for intervals $[a,b], [c,d]\subset (0,+\infty)$ we will write
$$
[a,b] \precsim [c,d] \iff a<c\ \mbox{and} \ b<d;\ \ \ \ \ \ \ 
[a,b] \prec [c,d] \ \iff\ b < c.
$$
For a symbol $\alpha\in \Ak$ we write $\alpha^k = \overbrace{\alpha\ldots\alpha}^{k}$, with $\alpha^0$ being the empty word.
Introduce the lexicographic ordering on $\{1,2\}^k$ for every $k\in \N$, so that
$$
1^k < 21^{k-1} < 121^{k-2} < \cdots < 12^{k-1} < 2^k.
$$

\begin{lem} \label{lem2}
For every $k\in \N$, if $v<w$, with $v,w\in \{1,2\}^k$, then 
$$
f_v(x) < f_w(x)\ \ \mbox{for all}\ x>0,
$$
and hence $I_{v3}^{(t)} \precsim I_{w3}^{(t)}$.
\end{lem}

Recall that $f_1$ and $f_2$ do not depend on $t$ and $f_3^{(t)}(x) = x/4 + t/2$.

\begin{proof}[Proof of Lemma~\ref{lem2}] It suffices to prove the claim for consecutive $v<w$, and hence by the definition of the lexicographic order it is enough to check that
$$
f_{2^m1}(x) < f_{1^m 2}(x) \ \ \mbox{for all}\ x>0\ \ \mbox{and}\ \ m\ge 0.
$$
We have $f_{2^m}(x) = 4^{-m}x$. To compute $f_{1^m}$ note that $f_{1^m} = f^m_1 = f_{A^m}$. One can check by induction that
$$
A^m = 
\begin{pmatrix}
2^{-m} & 0 \\
\frac{2^{m+2}(1-4^{-m})}{3} & 2^m \\
\end{pmatrix},
$$
hence
\be \label{eq1}
f_{1^m}(x) = \frac{x}{4^m\bigl(1 + 4x(1-4^{-m})/3\bigr)}.
\ee
Thus,
$$
f_{2^m1}(x) = \frac{x}{4^{m+1}(1+x)} < \frac{x}{4^{m+1}\bigl(1 + x(1-4^{-m})/3\bigr)} = f_{1^m 2}(x) \ \ \mbox{for all}\ x>0,
$$
as desired.
\end{proof}

\begin{lem} \label{lem3}
For every $k\in \N$, if $v<w$, with $v,w\in \{1,2\}^k$, then there exists $T>0$ such that 
$$
I_{v3}^{(t)} \prec I_{w3}^{(t)}\ \ \mbox{for all}\ t\ge T.
$$
\end{lem}

\begin{proof}[Proof of Lemma~\ref{lem3}]
In view of Lemma~\ref{lem2}, it is enough to show the claim for consecutive $v<w$ in the lexicographic order, which means that there exists $u\in \{1,2\}^*$, possibly
empty, such that $v = 2^m 1 u$ and $w = 1^m 2 u$. Recall that $I = [0,2t/3]$, hence $I_3 = [t/2,2t/3]$. Let $a = f_u(t/2)$ and $b = f_u(2t/3)$, so that
$I_{u3} = [a,b]$.
Using the calculations in the proof of the last lemma, we obtain that $I_{v3}^{(t)} \prec I_{w3}^{(t)}$ is equivalent to
$$
\frac{b}{4^{m+1}(1+b)} < \frac{a}{4^{m+1}\bigl(1 + a(1-4^{-m})/3\bigr)},
$$
which is the hardest to achieve when $m=0$, when it reduces to 
\be \label{want1}
\frac{b}{b+1} < a \iff b-a < ab.
\ee
Now there are two cases. If $u=2^\ell$, we simply have $a = 4^{-\ell}t/2$ and $b = 4^{-\ell}2t/3$, then $b-a = 4^{-\ell}t/6 < 4^{-2\ell}t^2/3 = ab$ for $t > 4^\ell/2$.
The more difficult case is when 
$$
u = U12^\ell,\ \ \ U \in \{1,2\}^*
$$
($U$ may be empty). Note that $\lim_{t\to +\infty} f_{2^\ell}(t/2) = \infty$. On the other hand, $f_1(x) < 1/4$ for all $x>0$, and
$$
\lim_{x\to +\infty} f_1(x) = \lim_{x\to +\infty} \frac{x}{4x+4}= \frac{1}{4}\,.
$$
Thus, for any $\delta>0$ we can find $T>0$ sufficiently large such that for all $t\ge T$ holds
$$
\frac{1}{4} - \delta < a':= f_{12^\ell}(t/2) < f_{12^\ell}(2t/3) =:b' < \frac{1}{4}\,.
$$
Then
$$
b-a = f_u(2t/3) - f_u(t/2) = f_U(b') - f_U(a').
$$
Notice that $f_1(x) > x/5$ for $x\in (0,1/4)$, and of course, $f_2(x) = x/4>x/5$ for $x>0$, hence
$$
b> 5^{-|U|}b',\ \ a> 5^{-|U|}a' \implies ab > 5^{-2|U|}a'b' > 5^{-2|U|-2},
$$
assuming $\delta < 1/20$. By the continuity of $f_U$, we can choose $\delta \in (0,1/20)$ such that 
$$
b-a = f_U(b') - f_U(a') < 5^{-2|U|-2} < ab,
$$
achieving \eqref{want1}. The lemma is proved.
\end{proof}

\begin{lem} \label{lem4}
For every $k\in \N$, if $v\in \{1,2\}^{k+1},\ w\in \{1,2\}^k$, then 
$$
I_{v3}^{(t)} \prec I_{w3}^{(t)}\ \ \mbox{for all}\ t\in (0, 3).
$$
\end{lem}

\begin{proof}[Proof of Lemma~\ref{lem4}]
In view of Lemma~\ref{lem2}, it is enough to consider $v = 2^{k+1}$ and $u = 1^k$. By \eqref{eq1}, 
$$
I_{2^{k+1}3}^{(t)} \prec I_{1^k3}^{(t)} \iff \frac{2t}{3\cdot 4^{k+1}} < \frac{t/2}{4^k\bigl(1 + 2t(1-4^{-k})/3\bigr)}\,,
$$
which is equivalent to
$$
t < \frac{3}{1-4^{-k}}\,.
$$
\end{proof}
Combining Lemmas~\ref{lem1}, \ref{lem3}, and \ref{lem4}, yields the proposition.
\end{proof}

\begin{prop}\label{prop:confdim} We have
$\dim_{\rm conf}( \Fk_t^{(n)})\to \dim_{\rm conf}( \Fk_t)$, as $n\to \infty$.
\end{prop}

This is standard, but we provide the proof for the reader's convenience.
Before the proof, we recall some general facts about the conformal dimension for hyperbolic IFS on the line. Let $\Ik$ be a finite alphabet, and let
$\Phi=\{\phi_i\}_{i\in \Ik}$ be a $C^{1+\theta}$-smooth hyperbolic IFS on a compact interval $J\subset \R$, that is, there exist $0 < \gam_1 < \gam_2 < 1$ such that 
$$
0 < \gam_1 \le |\phi_i'(x)| \le \gam_2 < 1\ \ \ \mbox{for all}\ i\in \Ik\ \ \mbox{and}\ x\in J.
$$
Let $d_n>0$ be such that
$$
\sum_{u\in \Ik^n} \|\phi'_u\|^{d_n}=1,
$$
where we use the sup-norm on $J$. It is well-known (see, e.g., \cite[Chapter 5]{Falconer_Tech}) that 
\be \label{eq:dimconf}
\dim_{\rm conf}( \Gk) = \lim_{n\to \infty} d_n,
\ee
and moreover,
$$
|\dim_{\rm conf}( \Phi) - d_n| \le O(1)\cdot n^{-1},
$$
where the constant depends only on the IFS. In fact, by the Chain Rule and the Principle of Bounded Distortion, there exists $C>1$, depending only on the IFS, such that
for all $\bi,\bj\in \Ik^*$ holds
\be \label{bdp}
C^{-1} \|\phi_\bi'\|\cdot \|\phi'_\bj\| \le \|\phi'_{\bi\bj}\| \le \|\phi_\bi'\|\cdot \|\phi'_\bj\|.
\ee
From the upper bound in \eqref{bdp} we  obtain $d_n \ge d_{2n} \ge \ldots \ge \dim_{\rm conf}( \Phi)$, and from the lower bound we get
$$
1 = \sum_{\bi,\bj\in \Ik^n} \|\phi'_{\bi\bj}\|^{d_{2n}} \ge C^{-1} \Bigl(\sum_{\bi\in \Ik^n} \|\phi'_\bi\|^{d_{2n}}\Bigr)^2 \implies
\sum_{\bi\in \Ik^n} \|\phi'_\bi\|^{d_{2n}} \le C^{1/2}.
$$
Then
$$
\sum_{\bi\in \Ik^n} \|\phi'_\bi\|^{d_{2n}+\delta} \le 1,\ \ \ \mbox{where}\ \ \gam_2^{\delta n} = C^{-1/2}, \ \ \delta = \frac{\log C}{2n\log (1/\gam_2)}\,,
$$
and so $d_n \le d_{2n} + \delta$. Repeating this argument we obtain
$$
d_{2^m n} \ge d_n - \frac{\log C}{2n\log (1/\gam_2)}\cdot \Bigl[n^{-1} + (2n)^{-1} + \cdots + (2^{m-1} n)^{-1}\Bigr] \ge  d_n - \frac{\log C}{n\log (1/\gam_2)}\,,
$$
hence
\be \label{eq:est1}
d_n - \frac{\log C}{n\log (1/\gam_2)} \le \dim_{\rm conf}( \Phi) \le d_n\ \ \ \mbox{for all}\ n\in \N.
\ee

\begin{proof}[Proof of Proposition~\ref{prop:confdim}] Recall that $\Ak = \{1,2,3\}$ and $\Fk_t = \{f_1,f_2,f_3^{(t)}\}$ (now the parameter $t$ is fixed).
For $N\in \N$ let $d_N>0$ be such that
$$
\sum_{u\in \Ak^N} \|(f_u^{(t)})'\|^{d_N} = 1,
$$
so that
$$
s:= \dim_{\rm conf}(\Fk_t) = \lim_{N\to \infty} d_N.
$$
It is clear that $s\in (\half,1)$, since $f_2$ and $f_3^{(t)}$ are similitudes of contraction ratio $1/4$, and hence already 
$$\sum_{u\in \{2,3\}^n} 
\bigl\|\bigl(f_u^{(t)}\bigr)'\bigr\|^{1/2} = 1.
$$
Fix $\eps>0$ and let $N\in \N$ be such that
$$
d_N \in [s, s+\eps).
$$

Recall that we consider the IFS $\Fk_t^{(N)}$, for $N\in \N$, defined in \eqref{def:ftn}. Let $\wt \Ak^{(N)}:= \Ak^N \setminus \{1,2\}^N$, so that
$
\Fk_t^{(N)} = \bigl\{f_u:\ u\in \wt \Ak^{(N)}\bigr\}.
$
It is immediate that $\bigl\|\bigl(f_u^{(t)}\bigr)'\bigr\|\le 4^{-|u|}$ for all $u$, hence  we have
$$
\sum_{u\in \wt\Ak^N} \|(f_u^{(t)})'\|^{d_N}  \ = \!\sum_{u\in \Ak^N\setminus \{1,2\}^N} \|(f_u^{(t)})'\|^{d_N} \ge 1 - 2^N (4^{-N})^{d_N} \ge 1 - 4^{-N\eps}\ge 1/2,
$$
for $N$ sufficiently large, and this implies
\be \label{eq:est2}
\sum_{u\in \wt\Ak^N} \|(f_u^{(t)})'\|^{d_N-(2N)^{-1}} \ge 1.
\ee
Let $s_n^{(N)}>0$ be such that
$$
\sum_{\bi\in (\wt\Ak^N)^n} \|(f_\bi^{(t)})'\|^{s_n^{(N)}}=1,
$$
so that $\dim_{\rm conf}(\Fk_t^{(N)}) = \lim_{n\to \infty} s_n^{(N)}$ by \eqref{eq:dimconf}, and \eqref{eq:est2} yields 
$$
s_1^{(N)} \ge d_N-(2N)^{-1}.
$$
 On the other hand, applying the inequality \eqref{eq:est1} for  $\Fk_t^{(N)}$, with $n=1$, we obtain
$$
s_1^{(N)} - \frac{\log C}{\log(4^N)} \le \dim_{\rm conf}( \Fk_t^{(N)}) \le s_1^{(N)},
$$
since $\max_{\bi \in \wt \Ak^{(N)}}  \|(f_\bi^{(t)})'\| = 4^{-N}$. Combining everything, we get that
$$
\bigl|\dim_{\rm conf}( \Fk_t^{(N)}) - \dim_{\rm conf}( \Fk_t)\bigr| \le \eps + \frac{1}{2N} + \frac{\log C}{N\log 4}\,,
$$
which yields the desired claim.
\end{proof}

\begin{proof}[Proof of Theorem~\ref{th:main}]
We only need to show the lower bound for $\dim(\Lam_t)$.
We have $\dim(\Lam_t)\ge \dim(\Lam_t^{(n)}$ for all $n\in \N$, where $\Lam_t^{(n)}$ is the attractor of the IFS $\Fk_t^{(n)}$, defined in \eqref{def:ftn}.
It was shown in Proposition~\ref{prop:nondegen} that $\Fk_t^{(n)}$ is non-degenerate on some interval, hence on  the entire $(0,+\infty)$.
In order to apply Theorem~\ref{th2} we need to restrict the IFS and the parameter set to a fixed bounded interval. We can fix $J=[t_0,t_1]\subset (0,\infty)$
arbitrarily and consider $\Fk_t^{(n)}$ as an IFS on $[-\eps,2t_1/3]$ for a small $\eps>0$.
It is a real-analytic family of IFS, depending real-analytically on  $t\in J$. By Theorem~\ref{th2}, it satisfies the SESC on $\{0\}$ for $t\in J\setminus \Ek_n$, with
$\Dh(\Ek_n)=0$, and then we can apply Theorem~\ref{th1} to conclude that $\dim(\Lam_t) = \dim_{\rm conf}( \Fk_t^{(n)})$ for $t\in J\setminus \Ek_n$
(here we use also that $\dim_{\rm conf}( \Fk_t^{(n)})<1$, since it is the conformal dimension of an IFS obtained my removing some maps from the
$n$-th interate of $\Fk_t$, and the latter has conformal dimension $\le \log 3/\log 4$).
The proof is concluded by an application of Proposition~\ref{prop:confdim}.
\end{proof}


\section{Concluding remarks}

\subsection{On the dimension of the natural measure}
The following remark is due to Bal\'azs B\'ar\'any. 
For a conformal IFS $\Phi = \{\phi_i\}_{i\in \Ik}$, consider 
the Gibbs measure $\nu_\Phi$ on the
symbolic space $\Ik^\N$ corresponding to the potential
$$
\bi\mapsto s\cdot \log|\phi_{i_1}'(\Pi_\Phi(\sig \bi))|,
$$
where $s=s(\Phi)$ is the conformal dimension of $\Phi$ and $\Pi_\Phi$ is the natural projection. This Gibbs measure on $\Ik^\N$ satisfies
$$
C^{-1}|I_{\bi|_n}|^s \le \nu_\Phi([\bi|_n]) \le C|I_{\bi|_n}|^s,
$$
for some constant $C>1$, for {\em every} $\bi\in \Ik^\N$, see, e.g., \cite[Chapter 9]{PUbook}. The push-forward of $\nu_\Phi$, that is, $\mu_\Phi=\nu_\Phi \circ \Pi_\Phi^{-1}$, 
is called the {\em natural measure} for $\Phi$.
If $\Phi$ satisfies the Strong Separation Condition, the local dimension of $\mu_\Phi$ is equal to $s$ at every point, and hence, the $L^q$-dimension if equal to $s$ for all $q\ge 1$ (see, e.g.,
\cite[Section 9.2]{PUbook}, \cite[Chapter 11]{Falconer_Tech} or \cite[Section 2.6]{BSSbook} for definitions). In our example, for the IFS $\Fk_t$, for any $t>0$, the local dimension of the
natural measure is strictly less than the conformal dimension $s(\Fk_t)$ at the common fixed point 0, since the number of cylinder intervals of level $n$ containing $0$, all of comparable length, 
grows exponentially with $n$. This implies, by a lemma of Shmerkin \cite[Lemma 1.7]{Shmerkin}, that there exists $q\in (1,\infty)$ such that the $L^q$-dimension of $\mu_\Phi$ is strictly less than
the conformal dimension $s(\Phi)$.

\subsection{Questions}
\begin{enumerate}
\item[{\bf 1.}] Is there $t>0$ such that $\Fk_t$ satisfies the strong exponential separation condition?
\item[{\bf 2.}] Is there an algebraic $t>0$ such that the semigroup $\{A,B,C_t\}^+$ is free? 
(Note that the positive answer for Question 2 implies the positive answer for Question 1.)
\item[{\bf 3.}] In the cases when the Hausdorff dimension of the attractor of $\Fk_t$ is equal to the conformal dimension, is this also the Hausdorff dimension of the natural measure?
\end{enumerate}


\medskip

\section{Appendix: on the freeness of the semigroup $\{A,B\}^+$}

Here we show the proof of the freeness, following \cite{CHK99}.
Recall that 
$A = \begin{pmatrix}
\half & 0 \\
2 & 2 \\
\end{pmatrix},\ 
B = \begin{pmatrix}
\half & 0 \\
0 &  2 \\
\end{pmatrix}.$
Let $R = \begin{pmatrix} 4 & 0 \\ 4/3 & 1 \\ \end{pmatrix}.$ The freeness of $\{A,B\}^+$ is equivalent to the freeness of 
$$
\left\{ 2RA^{-1}R^{-1},2RB^{-1}R^{-1}\right\}^+ = \left\{\begin{pmatrix}
4 & 0 \\
0 & 1 \\
\end{pmatrix},\ 
\begin{pmatrix}
4 & 0 \\
1 &  1 \\
\end{pmatrix}\right\}^+=: \{E,F\}^+.
$$
Suppose that there exist matrices $X,Y\in  \{E,F\}^+$ such that $XE = YF$. Then
$$
X = \begin{pmatrix} x_1 & 0 \\ x_2 & 1\end{pmatrix},\ \ \ Y = \begin{pmatrix} y_1 & 0 \\ y_2 & 1\end{pmatrix},\ \ \mbox{with}\ \ x_1,x_2, y_1, y_2\in \Z,
$$
whence
$$
4x_2 = XE(2,1) = YF(2,1) = 4y_2 + 1,
$$
which is a contradiction.

\bigskip

\noindent {\bf Acknowledgment.} I am grateful to Bal\'azs B\'ara\'ny for sharing the question with me, for helpful discussions, and especially for his remark in Section 3.

\medskip

\bibliographystyle{plain}

\bibliography{Solomyak_Takahashi}

\end{document}